\documentclass[10pt,twoside]{article}
\usepackage{amsmath,amssymb,amsthm,color}
\usepackage[T1]{fontenc}
\usepackage[latin1]{inputenc} 
\usepackage{a4wide}
\usepackage{amsmath,amsthm,amsxtra}
\usepackage{xparse}
\usepackage{tabularx} 
\usepackage{multicol}
\usepackage{color}
\usepackage{yfonts}
\usepackage{float}
\usepackage{graphicx}
\usepackage{subcaption}
\usepackage{pdfsync}
\usepackage{graphicx}

\usepackage[colorlinks=true]{hyperref}
\hypersetup{linkcolor=red,citecolor=blue,filecolor=dullmagenta,urlcolor=blue}

\newtheorem{theorem}{Theorem}[section]
\newtheorem{corollary}[theorem]{Corollary}
\newtheorem{lemma}[theorem]{Lemma}
\newtheorem{proposition}[theorem]{Proposition}
\newtheorem{claim}[theorem]{Claim}

\theoremstyle{remark}

\newtheorem{remark}{Remark}[section]

\newcommand{\deri}{\frac{d}{dt}}
\newcommand{\intR}{\int_\mathbb{R}}
\newcommand{\intp}{\int_0^{\infty}}
\newcommand{\im}{\text{Im}}
\newcommand{\re}{\text{Re}}
\newcommand{\R}{\mathbb{R}}

 \numberwithin{equation}{section}
 
\DeclareMathOperator{\sech}{sech}

\begin{document}

\title{Decay of small odd solutions for long range Schr\"odinger and Hartree equations in one dimension}
\author{Mar\'ia E. Mart\'inez\\ Departamento de Ingener\'ia Matem\'atica DIM \\ FCFM Universidad de Chile\\
{\tt maria.martinez.m@uchile.cl}\footnote{Partially funded by Conicyt Beca Doctorado Nacional no. 21192076, CMM Conicyt PIA AFB170001 and Fondecyt Regular 1150202 and 1191412. Part of this work was done while the author was visiting Universit\'e Paris-Saclay (Sud), Institut Mittag-Leffler and Universidad de Granada FisyMat, whose support is greatly acknowledged.}}
\date{\today}
\maketitle

\begin{abstract}
We consider the long time asymptotics of (not necessarily small) odd solutions to the nonlinear Schr\"odinger equation with semi-linear and nonlocal Hartree nonlinearities, in one dimension of space. We assume data in the energy space $H^1(\R)$ only, and we prove decay to zero in compact regions of space as time tends to infinity. We give three different results where decay holds: semilinear NLS, NLS with a suitable potential, and defocusing Hartree. The proof is based on the use of suitable virial identities, in the spirit of nonlinear Klein-Gordon models \cite{KMM2}, and covers scattering sub, critical and supercritical (long range) nonlinearities. No spectral assumptions on the NLS with potential are needed.
\end{abstract}

\tableofcontents


\section{Introduction}


In this paper our goal is to study the long time behavoir of \emph{small odd} global solutions of the one-dimensional nonlinear Schr\"odinger (NLS) and Hartree equations
	\begin{equation}\label{equation}
 		iu_t+u_{xx}=g(u), \quad (t, x) \in \mathbb{R}\times \mathbb{R}.
 	\end{equation}
In the Schr\"odinger case (see Ginibre-Velo \cite{GinibreVelo}, Cazenave-Weissler \cite{CW} and Cazenave \cite{Cazenave}), we shall assume that the nonlinearity takes the form
	\begin{equation}\label{nls}
		g(u)=\mu V(x)u+ f\left(|u|^2\right)u,
	\end{equation}
where the potential $V :\R \to \R$ is a \emph{Schwartz even function} and $f:\R \to \R$ is a function such that for $1<p<5$ ($L^2$ subcritical case), 
	\begin{equation}\label{hipotesisf}
		|f(s)|\lesssim s^{\frac{p-1}{2}}, 
	\end{equation}
and that satisfies that $f \circ s^2$ is locally Lipschitz continuous. In this context, we denote $F(s)=\int_0^s f(v) dv$, for all $s>0$, and 
	\[ G(u)= \frac {\mu }{2} \intR V(x) |u|^2 dx + {\frac12}  \intR F(|u|^2) dx. \]
 \par In the Hartree case, we have
	\begin{equation}\label{nlh} 	
		g(u)=\sigma \left( W*|u|^2\right)u, 	\quad G(u)= \frac{\sigma}{4}\intR \left( W*|u|^2\right)|u|^2 dx,
	\end{equation} 	
where $\sigma=\pm 1$ and the potential $W$ is given by
	\begin{equation}\label{nlh1}
	 W(x)=\frac{1}{|x|^a}, \quad \text{with }\quad 0<a<1.
	\end{equation}	
\par The equation \eqref{equation} is Hamiltonian, and it is characterized by having at least the following conservation laws: 
\begin{itemize}
	\item Mass:
	\begin{equation}\label{mass}
	M(u(t)):= \intR|u(t)|^2dx=M(u(0)),
	\end{equation}
	\item Energy:
	\begin{equation}\label{energy}
	E(u(t)):=\frac{1}{2} \intR |\nabla u(t)|^2 dx + G(u(t))=E(u(0)),
	\end{equation}
	\item Momentum:
	\begin{equation}\label{momentum}
	P(u(t)):=\im \intR u(t)\overline{u}_x(t)dx = P(u(0)).
	\end{equation}
\end{itemize}	

The NLS equation (\ref{equation})-(\ref{nls}) with nonlinearity $f(s)=\pm s^{\frac{p-1}{2}}$ is commonly known as the semilinear Schr\"odinger
equation \cite{Cazenave}. In particular, if $f(s)= -s^{\frac{p-1}{2}}$, we say that the equation is focusing, while the defocusing case takes place when $f(s)= s^{\frac{p-1}{2}}$. It is well-known that this one-dimensional semilinear Sch\"odinger equation is globally well-posed for initial data in $H^1(\mathbb R)$ when $1<p<5$, and blow up may occur if $p\geq 5$, see e.g. \cite{Glassey,MR} and subsequent works.

\medskip

On the other hand, the Hartree equation \eqref{equation} with \eqref{nlh} is also locally well-posed in $H^1(\R)$, and globally well-posed for small data, see \cite[Corollary 6.1.5]{Cazenave} for instance. This comes from the fact that the potential $W$ in \eqref{nlh1} is an even function that satisfies the following properties:
\begin{itemize}
	\item $W \in L^1(\R)+L^{\infty}(\R),$
	\item The function $(W*|u|^2)|u|^2$ is integrable. For the case \eqref{nlh1}, one has the estimate $$ \intR \left( |x|^{-a}*|u|^2\right)|u|^2 dx <\infty,$$
	(we prove this using the Hardy-Littlewood-Sobolev inequality \cite[Theorem 4.3, p. 106]{LiebLoss} with $p=r=\frac{2}{2-a}$).
\end{itemize}	
	This means that we are in the case of 
	 \cite[Example 3.2.11]{Cazenave} and \cite[Corollary 4.3.3]{Cazenave}, which implies the local well-posedness of the Hartree equation.
\par 
\medskip

In this paper we are interested in the asymptotic behavior of small solutions to \eqref{equation}, both in the NLS case (with and without potential), and in the nonlocal Hartree case, at least in the defocusing case. The literature on this subject is huge; we present now a (far from complete) account of the most relevant results. 
\medskip

\par It is known that for subcritical (in the sense of GWP and scattering) semilinear NLS equation ($f(s)= \pm s^{\frac{p-1}{2}}$, $3<p<5$), scattering to a free solution exists (see, for instance, Ginibre and Velo \cite{GinibreVelo}, Tsutsumi \cite{Tsutsumi0} and Nakanishi-Ozawa \cite{NO}). Nevertheless, in Strauss \cite{Strauss} and Barab \cite{Barab} it was proven that one cannot expect the same scattering for the critical $(p=3)$ and super critical case ($p<3$), and modified scattering is believed to occur. This was generalized recently by Murphy and Nakanishi \cite{MurphyNakanishi} for the semilinear NLS equation with potential and Hartree-type nonlinearities as \eqref{nlh1}. 

\medskip

\par 
Precisely, \emph{modified scattering} for $d$ dimensional critical NLS equation with nonlinearities
	\[
	g(u)=\sigma |u|^{p-1}u, \quad p=1+\frac2d, \quad d=1,2,3;
	\]
and the Hartree equation with Coulomb potential	
	\[
	g(u)=\sigma \left(|x|^{-1}*|u|^2\right)u, \quad d\ge2,
	\]
and small initial condition, was first proved by Ozawa \cite{Ozawa} and by Ginibre and Ozawa \cite{GO}. Moreover, it was shown that solutions $u$ of such equations present the decay
	\begin{equation}\label{0}
	 \|u(t)\|_{L^{\infty}} \lesssim (1+|t|)^{-d/2} ,
	\end{equation}
when the initial data is sufficiently small in weighted Sobolev spaces (see also Hayashi-Naumkin \cite{HN}, and Kato-Pusateri \cite{KP}, for instance). 
Through a thorough analysis of the solution profile, a simplified proof of scattering in the critical defocusing NLS and Hartree equations has been exhibited in \cite{KP}. 

\medskip

\par Similar recent results hold for the NLS case \emph{with a potential}, as was shown by Cuccagna, Visciglia and Georgiev \cite{CVG} for $p>3$, and Naumkin \cite{Naumkin} and Germain-Pusateri-Rousset \cite{GPR} for the critical case $p=3$ (see also \cite{GPR0}). Indeed, assuming that the potential $V$ is such that $-\frac 12 \Delta +V$ does not have negative eigenvalues nor resonances at zero, they were able to prove the decay \eqref{0} for solutions of subcritical ($p>3$) and critical ($p=3$) NLS equation in one dimension. However different the methods to prove this decay are from each other, it is not clear to us if they still hold by assuming less restrictive spectral conditions (for instance, the absence of resonances).

\medskip

Finally, following idea introduced in \cite{KMM1}, about considering odd data only, Delort \cite{Delort} proved modified scattering for small (smaller than a parameter $\epsilon$) odd solutions $u$ to \eqref{0} with data in $H^{0,1}\cap H^N$, $N$ large, and showed (among other things) the precise decomposition for large time
	\[
	u(t,x) = \frac{\epsilon}{\sqrt{t}}\textit{A}_{\epsilon} \left( \frac{x}{t} \right) \exp\left[	-i\frac{x^2}{2t}+i\epsilon^2 \log t \left|\textit{A}_{\epsilon}\left(\frac{x}{t}\right)\right|^2\right] +r(t,x),		
	\]
where the continuous function $\textit{A}_{\epsilon}$ is bounded in $L^2(\R) \cap L^{\infty}(\R)$, $\theta \in (0, \frac{1}{4})$
and 
	\[
	\|r(t,\cdot )\|_{L^{\infty}}=O(\epsilon t^{-\frac{3}{4}+\theta}), \quad \|\textit{A}_{\epsilon}(x)\langle tx \rangle^{-2}\|_{L^{\infty}}=O(\epsilon t^{-\frac{1}{4}+\theta}),	\]
	and
	\[
	\|r(t,\cdot )\|_{L^2}=O(\epsilon t^{-\frac{1}{4}+\theta}), \quad \|\textit{A}_{\epsilon}(x)\langle tx \rangle^{-2}\|_{L^{2}}=O(\epsilon t^{-\frac{5}{8}+\frac\theta2}).	\]
\par 
Notice that all positive decay/scattering results above mentioned cannot deal with the one dimensional NLS (for $p<3$) and Hartree equations. This is in part explained by the lack of precise nonlinear estimates in the case of long range nonlinearities.

\medskip

Our main goal in this paper is to extend in some sense the recently mentioned results \cite{Naumkin,GPR,KP,Delort} and show decay of small solutions to the above equations, regardless the (supercritical with respect to scattering) power of the nonlinearity. In particular, we consider nonlinearities NLS with $1<p<5$ and Hartree long range supercritical in one dimension. 

\medskip


Our first result covers the NLS case without potential ($1<p<5$).
 
\begin{theorem}\label{T1}
	Suppose $u(t) \in H^1(\mathbb{R})$ is a global odd solution of the equation \eqref{equation}-\eqref{nls} and $\mu=0$ such that, for some $\varepsilon>0$ small, 
		\begin{equation}\label{condition_s}
		\|u(t=0)\|_{H^1(\mathbb{R})}\le \varepsilon.
		\end{equation}
	Then, 
		\begin{equation}\label{result_s1}
		\lim_{t \to \infty } \Big( \|u(t)\|_{L^{2}(I)} + \|u(t)\|_{L^{\infty}(I)} \Big)=0,
		\end{equation}
	for any bounded interval $I \subset \R$. Moreover, if the equation is defocusing, the smallness condition \eqref{condition_s} is not needed.
\end{theorem}


\begin{remark}
	NLS \eqref{equation} preserves the oddness of the initial data along the flow.
\end{remark}

\begin{remark}
	Theorem \ref{T1} is sharp. Indeed, it is not true for $u(t)\in H^1$ even. A simple counterexample in this case is the nondecaying soliton itself:
		\begin{equation}\label{Solitary_wave}
		u(t,x)= Q_c(x)e^{ict}, \qquad 0<c\ll 1, 
		\end{equation}
	and $Q_c>0$ solving $Q_c''-cQ_c +Q_c^p=0$, $Q_c\in H^1$. Note that this solution is even in space and small in $H^1$ provided $c\ll1$. Also, the Satsuma-Yajima breather solutions (see \cite{SY} and \cite[eqn. (1.16)]{AFM}) are arbitrarily small nondecaying even solutions to NLS \eqref{nls} in the integrable \cite{ZS} case $p=3$.
\end{remark}

\begin{remark} 
	For an interval $I=I(t)$ growing in time, Theorem \ref{T1} is also sharp. Indeed, see the works \cite{MManihp,Vinh} for the construction of odd solutions composed of two solitary waves with non zero speeds for finite time. These asymptotic 2-soliton solutions can be arbitrarily small in the energy space, but they separate each other as time evolves, leaving any compact region in space for sufficiently large time. In this sense, these solutions do not contradict Theorem \ref{T1}. 
\end{remark}

\begin{remark}
	From the identity
		\[
		\frac{d}{dt} \intR x |u(t,x)|^2dx =-2\, \im \intR u(t)\overline{u}_x(t)dx= -2P(u(t)),
		\]
	valid if $xu(t=0)\in L^2$, we can see that nontrivial, nondecaying periodic-in-time solutions (i.e. breathers) of NLS may exist only if their momentum vanishes. See \cite{MP} for more details on these properties of breather solutions.
\end{remark}

\begin{remark}
	Sometimes, instead of assuming odd data, the additional assumption $\| xu(t=0)\|_{L^2}\ll 1$ is considered. This condition works with even data, and rules out the existence of small solitary waves as in  \eqref{Solitary_wave}, since small solitary waves satisfy $\| xQ_c\|_{L^2}\gg 1$. 
\end{remark}

\begin{remark}
Note that \eqref{result_s1} does not contain the $\dot H^1$ norm of the solution. This is  a standard open issue in the field, see e.g. \cite{Delort} for similar results. In our case, the lack of control on the decay of this semi-norm is due to the emergency of uncontroled $H^2$ terms in the dynamics of the energy norm. 
\end{remark}

\bigskip

The proof of Theorem \ref{T1} is based on the introduction of a virial identity adapted to the NLS dynamics. Following the ideas presented in \cite{KMM2,KMM1}, which considered the nonlinear Klein-Gordon case, we use here a functional adapted to the momentum \eqref{momentum}. Once this virial identity is established, decay is proved in a standard form. 

\medskip

Using inverse scattering techniques, Deift and Zhou \cite{DZ} described the asymptotic behavior of solutions of the defocusing, nearly integrable quintic perturbation of cubic NLS
	\begin{equation}\label{CQ}
	iu_t+u_{xx}=|u|^2 u +\epsilon |u|^4u, \quad \epsilon>0.
	\end{equation}
The case $\epsilon<0$ was left open in \cite{DZ}. Using the techniques of this paper, we are able to give a partial answer for this remaining case:

\begin{corollary}
	Let $\epsilon\neq 0$, and let $u\in C(\R;H^1(\R))$ be a global small odd solution of \eqref{CQ}. Then  \eqref{result_s1} is satisfied.
\end{corollary}

The proof of this result immediately follows from Theorem \ref{T1}. 

\medskip

Our second result deals with NLS \eqref{equation} with nonzero potential in \eqref{nls}. In this case, we also provide time decay results in the case $\mu V$ small and spatially decaying fast enough, complementing \cite{CVG,Delort,Naumkin,GPR}.

\begin{theorem}[NLS with potential]\label{T1a}
	Assume $V\neq 0$ even as in \eqref{nls}. Under the assumptions of Theorem \ref{T1}, suppose additionally that $V$ satisfies
		\begin{equation}\label{hipotesisV}
		\intR \left( |V(x)| +|V'(x)| \right) \cosh(2x) dx <+\infty.
		\end{equation}
	Then there exists $\mu_0>0$ such that for all $\mu\in (0,\mu_0)$,  \eqref{result_s1} holds 
	for any bounded interval $I \subset \R$.
\end{theorem}

\begin{remark}
	Note that Theorem \ref{T1a} does not require that the operator $-\partial_x^2 \pm \mu V$ satisfies specific spectral properties as in \cite{CVG,Naumkin,GPR}; only the decay hypothesis \eqref{hipotesisV} is needed. In particular, no \emph{nonresonance} condition is needed for having \eqref{result_s1}. This fact reveals that the non resonance condition is essentially linked to the evenness of the involved data.
\end{remark}


\begin{remark}
We can ask for $V$ decaying slower than in \eqref{hipotesisV}, but proofs are probably more complicated; we hope to consider this problem elsewhere.
\end{remark}

\medskip

Finally, we deal with the Hartree case.

\begin{theorem}[Defocusing Hartree equation]\label{TH1}
	Suppose that $u \in H^1(\mathbb{R})$ is a global odd solution of equation \eqref{equation} with \eqref{nlh1} and $\sigma=1$.
	Then
		\begin{equation}\label{result_h}
		\lim_{t \to \infty }(\|u(t)\|_{L^{2}(I)}+ \|u(t)\|_{L^{\infty}(I)})=0,
		\end{equation}
	for any bounded interval $I \subset \R$.
\end{theorem}

\begin{remark}
	Theorem \ref{TH1} proves the non-existence of odd standing waves solutions for the equation \eqref{equation} with defocusing Hartree type non-linearities. 
\end{remark}

\begin{remark}
	Theorem \ref{TH1} does not include the focusing case, which is an open problem of independent interest. In that sense, the scattering problem for the $d\geq 2$ generalized Hartree equation was recently treated in Arora-Roudenko \cite{AR}. 
\end{remark}

\begin{remark}
	Focusing Hartree equation \eqref{equation} with \eqref{nlh1} ($\sigma=-1$) admits solitary waves solutions (or solitons) 
			\[u(t, x)=e^{ict}Q_c(x) \in H^1\]
	where $Q_c:\R \to \R$ is an $H^1$-solution of the Choquard equation
		\begin{equation}\label{choquard}
		 \Delta Q+ \left( \frac{1}{|x|^{a}}*|Q|^p\right)Q - \lambda Q=0, \quad c\in \R.
		 \end{equation}
	These solutions are, up to translation and inversion of the sign, positive and radially symmetric functions \cite{CSS, VanSchaftingen}. Moreover, solitary waves for the focusing Hartree equation are stable, as was proven by Cazenave and Lions in \cite{CL}. See also Ruiz \cite{Ruiz} for more details on solitary waves for Hartree.
\end{remark}


%

\begin{remark}[NLS around solitary waves]
Solitary waves in mass subcritical NLS exist and they are stable. The first results on stability were provided by Cazenave and Lions in \cite{CL}, where orbital stability of solitary waves for the NLS equation \eqref{equation}-\eqref{nls} without potential was proven (see also \cite{Weinstein, GSS}). Stability of several NLS solitons well-decoupled was proved in \cite{MMT2}, and in \cite{Kap} for the integrable case. The asymptotic stability for the same equation was studied by Buslaev and Pere'lman in \cite{BuslaevPerelman} in the supercritical regime; this result was later generalized by Cuccagna in \cite{Cucc1, Cucc2, CM} for dimensions $d\ge 3$, and under special spectral conditions on the linealized operator around the solitary wave. The one dimensional case, under similar spectral assumptions and even data perturbations of the standing wave, was studied by Buslaev and Sulem \cite{BuslaevSulem}. For the NLS equation with potential \eqref{equation}-\eqref{nls}, results for asymptotic stability of gound states (also, under spectral conditions) were provided by Soffer and Weinstein in \cite{SW1, SW2}, see also \cite{SW}. We believe that some of the ideas in this paper can be generalized to the case of asymptotic stability for solitary waves, but with harder proofs. See e.g. the recent paper by Cuccagna and Maeda \cite{CuccagnaMaeda}, and the NLKG paper by Kowalczyk, Martel and Mu\~noz \cite{KMM3}. 
\end{remark}



\subsection*{Notation} To simplify the notation we will denote $u_1= \text{Re } u$, $u_2=\text{Im }u$. Let $\alpha(x)\geq 0$ be a weight. We also denote by 
	\begin{equation}\label{weighted}
	\|u(t)\|_{H^1_{\alpha}(\R)}^2 := \intR \alpha(x) \left(|u_x(t,x)|^2 + |u(t,x)|^2\right)dx,
	\end{equation}
the weighted $H^1$-norm with weight $\alpha$.

\subsection*{Organization of this paper} This paper is written as follows. In Section \ref{withoutpotential} we prove Theorem \ref{T1}, NLS without potential. Section \ref{withpotential} is devoted to the proof of Theorem \ref{T1a}, namely NLS with potential. Finally, Section \ref{proof_Hartree} deals with the Hartree case (Theorem \ref{TH1}).

\bigskip

\section{Schr\"odinger equation without potential}\label{withoutpotential}

In this Section we prove Theorem \ref{T1}. Consider the equation (\ref{equation}) with (\ref{nls}) and $V \equiv 0$. That is, 
	\begin{equation}\label{NLS1}
	 iu_t+u_{xx}= f\left(|u|^2\right)u,\qquad u\in H^1 \text{ odd.}
	\end{equation}
As claimed in the introduction, the proof here follows the ideas in \cite{KMM2}, with some minor differences.

\subsection{A virial identity}\label{VIwithoutpotential}

We shall introduce a standard virial identity adapted to \eqref{NLS1}. Let $\varphi \in C^{\infty}(\mathbb{R})$ be bounded and to be chosen later, $u(t) \in H^1(\R)$ a solution of equation (\ref{NLS1}) and define
	\begin{equation}\label{I}
	I(u(t)):=\im\intR \varphi(x) u(t,x)\overline{u}_x(t,x)dx.
	\end{equation}
Then we have the following:

\begin{lemma} \label{VirialIdentity1}
	For $u\in C(\R;H^1(\R))$ one has $I(u(t))$ well-defined and bounded in time. Moreover, we have the virial identity
		\begin{equation}\label{virial1}
		\begin{aligned}
		-\deri I(t)= 2 \intR \varphi_x |u_x|^2 dx - \frac{1}{2}\intR \varphi_{xxx} |u|^2 dx- \intR \varphi_x \left[F\left(|u|^2\right)-f(|u|^2)|u|^2\right] dx.
		\end{aligned}
		\end{equation}
\end{lemma}

\begin{proof} 
	Let $u(t) \in H^1(\R)$ such that it satisfies equation (\ref{NLS1}). Then, we integrate by parts
		\begin{align*}
		\deri I(u(t))
		&= \im \intR \varphi u_t \overline{u}_x dx + \im \intR  \varphi u \overline{u}_{xt} dx\\
		&= \im \intR \varphi u_t \overline{u}_x dx - \im \intR \left(\varphi u \right)_x  \overline{u}_{t} dx.
		\end{align*}
	Then,
		\begin{align*}
		\deri I(u(t))
		&=- \im \intR i \varphi\left( iu_t\right) \overline{u}_x dx - \im \intR i \left(\varphi u\right)_x \overline{i u_t} dx\\
		&=- \re \intR  \varphi\overline{iu_t} u_x dx - \re \intR  \left(\varphi u\right)_x \overline{i u_{t}} dx.
		\end{align*}
	Computing the derivative on the second term above
		\begin{equation}\label{a1}
		\begin{aligned}
		\deri I(u(t))
		&= - 2\re \intR \varphi \overline{i u_t} u_x dx - \re \intR \varphi_x u \overline{i u_t} dx \\
		&= - 2\re \intR \varphi (i u_t) \overline{u}_x dx - \re \intR \varphi_x \left(i u_t\right)\overline{u} dx.
		\end{aligned}
		\end{equation}
	Thus, using (\ref{NLS1}), we get
		\begin{align*}
		\deri I(u(t))= 
		& ~ {}2\re \intR \varphi u_{xx} \overline{u}_x dx + \re \intR \varphi_x  u_{xx} \overline{u} dx \\ 
		& ~ {} - 2\re \intR \varphi f\left(|u|^2\right)u \overline{u}_x dx - \re \intR \varphi_x f\left(|u|^2\right)u\overline{u} dx.
		\end{align*}
	We notice that $2 \re \left(u_x \overline{u}\right)=2 \re\left( u \overline{u}_x\right) = \left(|u|^2\right)_x$, then
		\begin{align*}
		\deri I(u(t))= 
		& ~ {} \intR \varphi \left(|u_x|^2\right)_x dx + \re \intR \varphi_x  u_{xx}\overline{u} dx \\ 
		& ~ {} - \intR \varphi f\left(|u|^2\right)\left(|u|^2\right)_xdx - \intR \varphi_x f\left(|u|^2\right)|u|^2 dx.
		\end{align*}
	Recall the definition of $F(s)=\int_0^s f(v)dv$, which implies that $\left(F(s)\right)_x= f(s)s_x$. Furthermore,
		\begin{align*}
		\deri I(u(t))= 
		& ~ {} \intR \varphi \left(|u_x|^2\right)_x dx + \re \intR \varphi_x \overline{u} u_{xx} dx \\
		& ~ {}-  \intR \varphi \left( F\left(|u|^2\right)\right)_xdx -\intR \varphi_x f\left(|u|^2\right)|u|^2 dx.
		\end{align*}
	Integrating by parts, we obtain 
		\begin{align*}
		\deri I(u(t))= 
		& -2 \intR \varphi_x |u_x|^2 dx - \re \intR \varphi_{xx} \overline{u} u_{x} dx  
		+ \intR \varphi_x \left[F\left(|u|^2\right)-f\left(|u|^2\right)|u|^2\right]dx \\
		=& -2 \intR \varphi_x |u_x|^2 dx - \frac{1}{2}\intR \varphi_{xx} \left(|u|^2\right)_x dx  + \intR \varphi_x \left[F\left(|u|^2\right)-f\left(|u|^2\right)|u|^2\right]dx. 
		\end{align*}
	We integrate by parts again on the second term to obtain
		\begin{align*}
		\deri I(u(t))=& -2 \intR \varphi_x |u_x|^2 dx + \frac{1}{2}\intR \varphi_{xxx} |u|^2 dx  +  \intR \varphi_x \left[F\left(|u|^2\right)-f\left(|u|^2\right)|u|^2\right]dx.
		\end{align*}
\end{proof}

\subsection{Analysis of a bilinear form}\label{bilinearanalisys1} 

With the identity \eqref{virial1} in mind, we define the bilinear form 
	\begin{equation}\label{bilinear}
	B(w)= 2 \intR \varphi_x {w}_x^2 dx - \frac{1}{2}\int _{\mathbb{R}} \varphi_{xxx} w^2 dx, \quad w=u_i, \quad i=1,2.
	\end{equation}
	Here, $u=u_1+iu_2$, with $u_1,u_2$ real-valued.
	
	\medskip
	
Let $\lambda \in (1, \infty)$. As we explained before, our intention is to prove some estimation of $B$ using the weighted $H^1_\alpha$-norm introduced in \eqref{weighted}.
To obtain this, we will consider $\varphi(x)=\lambda \tanh \left(\frac{x}{\lambda}\right)$ on the virial identity (\ref{virial1}) and define the auxiliar function $\alpha(x)=\sqrt{\varphi_x(x)}$. Now, we estimate each term of the bilinear form $B$:
	\begin{align*}
	\intR \left(\alpha w\right)_x^2 dx 
	& = \intR \alpha^2 \left({w}_x\right)^2 dx +2 \intR \alpha \alpha_x w {w}_x dx  + \intR \left(\alpha_x\right)^2 w^2 dx\\
	& = \intR \varphi_x \left({w}_x\right)^2 dx + \intR \alpha \alpha_x \left(w^2\right)_x dx  + \intR \left(\alpha_x\right)^2 w^2 dx\\	
	& = \intR \varphi_x \left({w}_x\right)^2 dx - \intR \alpha \alpha_{xx}w^2 dx, 
	\end{align*}
using integration by parts in the last equality. Thus 
	\begin{align}\label{b11}
	\intR \varphi_x \left({w}_x\right)^2 dx=
	\intR \left(\alpha w\right)_x^2 dx + \intR  \frac{\alpha_{xx}}{\alpha}\left(\alpha w\right)^2 dx.
	\end{align}
Furthermore, noticing that $\varphi_{xxx}=\left(\alpha^2\right)_{xx}=2\left(\alpha \alpha_{xx} + \alpha_x^2\right)$, we get
	\begin{align}\label{b12}
	\intR \varphi_{xxx} w^2 dx = 2\intR\left(\frac{\alpha_{xx}}{\alpha}+\frac{\alpha_x^2}{\alpha^2}\right)\left(\alpha w \right)^2 dx.
	\end{align}
Hence, from (\ref{b11}) y (\ref{b12}), 
	\begin{align*}
	B(w)=2\intR \left(\alpha w\right)^2_x dx - \intR\left( \frac{\alpha_x^2}{\alpha^2}-\frac{\alpha_{xx}}{\alpha}\right)\left(\alpha w \right)^2 dx.
	\end{align*}
Since $\alpha(x)=\sech \left(\frac{x}{\lambda}\right)$, then
	\begin{align*}
	&\alpha_x(x)=-\frac{1}{\lambda} \sech\left(\frac{x}{\lambda}\right)\tanh\left(\frac{x}{\lambda}\right)\\ &\alpha_{xx}(x)=\frac{1}{\lambda^2}\left(\sech\left(\frac{x}{\lambda}\right)\tanh\left(\frac{x}{\lambda}\right)- \sech^3\left(\frac{x}{\lambda}\right)\right)
	\end{align*}
which implies that 
	\begin{align*}
	B(w) = 2\intR \left(\alpha w\right)^2_x dx - \frac{1}{\lambda^2} \intR\sech^2\left(\frac{x}{\lambda}\right)\left(\alpha w \right)^2 dx.
	\end{align*}
In order to prove Theorem \ref{T1} we need to prove that the bilineal part of (\ref{virial1}) is coercitive in some way. To be more precise, we would like the following 
	\begin{equation}\label{b4}
	B(w)\ge \intR \left(\alpha w\right)_x^2dx.
	\end{equation}
We introduce the auxiliar function $v=\alpha w$. Then we can set 
	\[\mathcal{B}(v) = 2\intR {v}^2_x dx - \frac{1}{\lambda^2} \intR\sech^2\left(\frac{x}{\lambda}\right)v^2 dx\]
\noindent so that 
	\[\mathcal{B}(v)=B(w).\]
This way, coercivity of the operator $\mathcal{B}$ implies (\ref{b4}). We recall now

\begin{proposition}[See \cite{KMM2}]\label{coercividad1}
	Let $v \in H^1(\R)$ be odd, $\lambda > 0$. Then 
		\begin{equation}
		\mathcal{B}(v)\ge \frac32 \intR v_x^2 dx.
		\end{equation}
\end{proposition}

\begin{proof}[Sketch of proof]
	We write
		\[\mathcal{B}(v) = \frac32\intR {v}^2_x dx + \frac 12 \left(\intR {v}^2_x dx - \frac{2}{\lambda^2} \intR\sech^2\left(\frac{x}{\lambda}\right)v^2 dx\right).\]
	Notice that
		\[-\frac{d^2}{dx^2} - \frac{2}{\lambda^2} \sech^2\left(\frac{x}{\lambda}\right)\]
	has only one negative eigenvalue corresponding to an even eigenfunction. This comes from the fact that (see \cite[Exercise 12]{Goldman}) the index of the operator 	
		\[-\frac{h^2}{\nu}\frac{d^2}{dx^2}-\gamma \sech^2\left(\frac{x}{a}\right)\]
	is equal to the largest integer $N$ such that 
		\[N<\frac{1}{2}\sqrt{8 \gamma \nu a^2 h^{-2}+1}-\frac{1}{2}.\]
	Since $v$ is odd, 
		\begin{equation}\label{b15}
		\intR {v}^2_x dx - \frac{2}{\lambda^2} \intR\sech^2\left(\frac{x}{\lambda}\right)v^2 dx \ge 0,
		\end{equation}
	and then (\ref{coercividad1}) holds.
\end{proof}

\subsection{Estimates of the terms on (\ref{virial1})}

\begin{lemma} \label{estimation2}
	Let $u \in H^1(\R)$ be odd. Then for some $C>0$, $u=u_1+iu_2$,
		\begin{equation}
		\|u\|^2_{H^1_\alpha(\R)} \le C\left(B(u_1)+B(u_2)\right).
		\end{equation}
\end{lemma}

\begin{proof} 
	We take $\lambda=100$. First, notice that from (\ref{b15}), we have 
		\begin{align*}
		\intR \left(\alpha w\right)_x^2\ge \frac{2}{100^2}\intR \sech^2\left(\frac{x}{100}\right)\left( \alpha w\right)^2 dx.  
		\end{align*}
	Using Proposition \ref{coercividad1}, this implies that 
		\begin{equation}\label{e0}
		B(w)\ge \frac{3}{2}\intR \left(\alpha w\right)_x dx \gtrsim \intR \sec^4\left(\frac{x}{100}\right) u^2 dx \gtrsim\intR \sech(x) w^2dx. 
		\end{equation}
	Thus, 
		\begin{equation}\label{cota11}
		\intR \sech(x) u_i^2dx \lesssim B(u_i), \quad i=1,2. 
		\end{equation}
	On the other hand, 
		\begin{align*}
		\intR \left(\alpha w\right)_x^2 dx 
		&\gtrsim \intR \alpha^2 \left(\alpha w\right)_x^2 dx\\
		&= \intR \alpha^4 {w}_x^2 dx+\intR \alpha^3\alpha_x \left(w^2\right)_x dx
		+ \intR \alpha^2\alpha_x^2 w^2 dx.
		\end{align*}
	We integrate by parts, 	
		\begin{align*}
		\intR \left(\alpha w\right)_x^2 dx 
		&\gtrsim \intR \alpha^4 {w}_x^2 dx-\intR \left(\alpha^3\alpha_x\right)_x w^2 dx
		+ \intR \alpha^2\alpha_x^2 w^2 dx\\
		& = \intR \alpha^4 {w}_x^2 dx-\intR \left(2 \alpha^2 \alpha^2_x + \alpha^3 \alpha_{xx}\right) w^2 dx.
		\end{align*}
	Then, from the definition of $\alpha$,	
		\begin{align*}
		\intR \left(\alpha w\right)_x^2 dx 
		&\gtrsim \intR \sech(x) {w}_x^2 dx-\intR \sech^4\left(\frac{x}{100}\right) w^2 dx.
		\end{align*}
	In other words, 
		\begin{equation*}
		\intR \sech(x) {w}_x^2 dx \lesssim \intR \left(\alpha w\right)_x^2 dx + \intR \sech^4\left(\frac{x}{100}\right) w^2 dx.
		\end{equation*}
	Then, from \eqref{e0}, we have that	
		\begin{equation}\label{e1}
		\intR \sech(x) {w}_x^2 dx \lesssim \intR \left(\alpha w\right)_x^2 dx.
		\end{equation}
	Hence, using Proposition \ref{coercividad1},
		\begin{align}\label{cota12}
		\intR \sech(x) {u_i}_x^2 dx \lesssim  B(u_i), \quad i=1, 2.
		\end{align}
	Finally, from (\ref{cota11}) y (\ref{cota12}), we get
		\[\|u(t)\|_{H^1_\alpha(\R)}^2 \lesssim B(u_1)+B(u_2).\]
\end{proof}

\begin{lemma} \label{estimation1}
	There exists $\varepsilon>0$ such that for every odd solution $u$ of (\ref{NLS1}) satisfying (\ref{condition_s})
	then
		\begin{equation}\label{cotaI1}
		-\deri I (u(t))\ge C \|u(t)\|^2_{H^1_\alpha(\R)}.
		\end{equation}
	where $C>0$. 
\end{lemma}

\begin{proof}
	Recall from \eqref{virial1} and the analysis of the previous section that 	
		\begin{equation*}
		\begin{aligned}
		-\deri I(u(t))= 
		&~{} 2 \intR \varphi_x |u_x|^2 dx - \frac{1}{2}\intR \varphi_{xxx} |u|^2 dx- \intR \varphi_x \left[F\left(|u|^2\right)-f(|u|^2)|u|^2\right] dx\\
		&~{}= B(u_1)+ B(u_2) - \intR \varphi_x \left[F\left(|u|^2\right)-f(|u|^2)|u|^2\right] dx.
		\end{aligned}
		\end{equation*} 
	Consecuently, in order to complete the proof, we need to control the remaining terms of (\ref{virial1}), since the terms involving the bilinear form $B$ have already been estimated by Lemma \ref{estimation2}. 
	\par Note that 
		\[\left|F\left(|u|^2\right)-f(|u|^2)|u|^2 \right| \lesssim |u|^{p+1}.\]
	Since $u$ is odd, 
		\begin{align*}
		\intR \sech^2\left(\frac{x}{\lambda}\right)|u|^{p+1}dx 
		& =2 \intp \sech^2\left(\frac{x}{\lambda}\right)|u|^{p+1}dx \\
		& = 2 \intp		 	\sech^{-(p-1)}\left(\frac{x}{\lambda}\right)\sech^{p+1}\left(\frac{x}{\lambda}\right)|u|^{p+1} \\
		& \simeq \intp  e^{(p-1)x/\lambda}
		\sech^{p+1}\left(\frac{x}{\lambda}\right)|u|^{p+1}dx.
		\end{align*}
	\par With a slight abuse of notation, set $v(t,x):=\sech\left(\frac{x}{\lambda}\right)u(t,x)$.
	Note that $v(t,0)=0$ and vanishes at infinity $\forall t \in \R$. Then, integrating by parts, 
		\begin{align*}
		\intp  e^{(p-1)x/\lambda} |v|^{p+1}dx 
		&=-\frac{\lambda}{p-1} \intp  e^{(p-1)x/\lambda}
		\left(|v|^{p+1}\right)_x dx \\
		&=-\frac{\lambda(p+1)}{p-1}\re \intp e^{(p-1)x/\lambda} |v|^{p-1} \bar v{v}_x dx.
		\end{align*}
	Hence,
		\begin{align*}
		\intp  e^{(p-1)x/\lambda} |v|^{p+1}dx 
		& =-\frac{\lambda(p+1)}{p-1}\re \intp e^{(p-1)x/2\lambda} |v|^{\frac{p-1}{2}}\overline{v}v_x\left(e^{(p-1)x/2\lambda}|v|^{\frac{p-1}{2}}\right) dx \\
		& \lesssim \|u\|^{(p-1)/2}_{L^{\infty}(\R)} \re \intp e^{(p-1)x/2\lambda} |v|^{\frac{p-1}{2}}\overline{v}v_x dx \\
		&\lesssim \|u\|^{(p-1)/2}_{L^{\infty}(\R)} \intp e^{(p-1)x/2\lambda} |v|^{\frac{p-1}{2}}|v||v_x| dx\\
		& = \|u\|^{(p-1)/2}_{L^{\infty}(\R)} \intp e^{(p-1)x/2\lambda} |v|^{\frac{p+1}{2}}|v_x| dx.
		\end{align*}
	By Young's inequality,
		\begin{align*}
		\intp  e^{(p-1)x/\lambda}
		|v|^{p+1}dx & \lesssim \|u\|^{p-1}_{L^{\infty}(\R)} \intp |v_x|^2 dx+\intp e^{(p-1)x/\lambda} |v|^{p+1}dx \\
		&\simeq \|u\|^{p-1}_{L^{\infty}(\R)} \intp |v_x|^2 dx+\intp \sech^{p-1}\left(\frac{x}{\lambda}\right)\sech^{p+1}\left(\frac{x}{\lambda}\right)|u|^{p+1}dx \\
		& = \|u\|^{p-1}_{L^{\infty}(\R)} \intp |v_x|^2 dx+\intp \sech^{2}\left(\frac{x}{\lambda}\right)|u|^{p+1}dx,
		\end{align*}
	which actually means that
		\[\intR \sech^2\left(\frac{x}{\lambda}\right)|u|^{p+1}dx \lesssim \|u\|^{p-1}_{L^{\infty}(\R)} \intp |\left(\alpha u\right)_x|^2 dx.\]
	By Sobolev's embedding,
		\[\intR \sech^2\left(\frac{x}{\lambda}\right)|u|^{p+1}dx \lesssim \|u\|^{p-1}_{H^1(\R)} \intp |\left(\alpha u\right)_x|^2 dx.\]
	Now, it is a fact that for every $0<\varepsilon<1$, there exists $\delta(\varepsilon)$ such that $\|u(0)\|_{H^1(\R)}\le \delta(\epsilon)$ implies that $\displaystyle \sup_{t \in \R} \|u\|_{H^1(\R)}< \varepsilon$ (see \cite[Corollary 6.1.4]{Cazenave} or the conservation of energy and mass \eqref{energy}-\eqref{mass}). This way, from Proposition \ref{coercividad1}, we get
		\[\intR \sech^2\left(\frac{x}{\lambda}\right)|u|^{p+1}dx \lesssim \varepsilon^{p-1} \left(B(u_1)+B(u_2)\right).\]
	So, choosing $\varepsilon$ sufficiently small, (\ref{cotaI1}) is proved.
\end{proof}

\begin{remark}[Defocusing case]\label{defocusing1}
	Note that in the semilinear defocusing case $f(|u|^2)=|u|^{p-1}$,
		\[ F\left(|u|^2\right)-f(|u|^2)|u|^2 = \left(\frac{2}{p+1} - 1 \right)|u|^{p+1}.  \]
	Since $p>1$, $\frac{2}{p+1}|u|^{p+1} - 1<0$ which means that the remaining term on \eqref{virial1} involving the nonlienarity is positive: 
		\[ - \intR \varphi_x \left[F\left(|u|^2\right)-f(|u|^2)|u|^2\right] dx \ge 0 \]
	and then Lemma \ref{estimation2} is enough to conclude Lemma \ref{estimation1}.	
\end{remark}	

With this estimation, we can now prove the key to get Theorem \ref{T1}.

\begin{proposition}\label{propositionkey}
	There exists a constant $C>0$ such that 
		\begin{equation}\label{key1}
		\intp \|u(t)\|_{H^1_\alpha(\R)}^2 dt \le C\varepsilon^2.
		\end{equation}
\end{proposition}

\begin{proof}
	Let $\tau>0$. We integrate (\ref{cotaI1}) over $[0,\tau]$ 
		\begin{equation*}
		\int_0^{\tau} \|u(t)\|^2_{H^1_\alpha(\R)} dt \le C\left(I(u(0))-I(u(\tau))\right) \le CI(u(0))
		\end{equation*}
	From H\"older inequality and (\ref{condition_s}) we get that
		\[I(u(0)) \le \|u(0)\|_{L^2(\R)}\|u_x(0)\|_{L^2(\R)}\le \varepsilon^2.\]
	This last fact implies that	
		\begin{equation*}
		\int_0^{\tau} \|u(t)\|^2_{H^1_\alpha(\R)} dt \le C \varepsilon^2.
		\end{equation*}
	Now, taking $\tau \to \infty$, we conclude the proof.	
\end{proof}

\subsection{End of proof of Theorem \ref{T1}}\label{proof1} Now Theorem \ref{T1} is ready to be proved:

\bigskip 

\textit{Step 1: The $L^2$ norm tends to zero:} Let $\varphi \in C^{\infty}(\R)$ be bounded. Then we compute
	\begin{align*}
	\deri \left(\frac{1}{2}\intR \varphi|u(t)|^2 dx\right) 
	&=\re \intR \varphi\overline{u} u_t dx =-\re \intR i \varphi\overline{u}  \left(iu_t\right)  dx\\
	&= \im \intR \varphi \overline{u} \left(iu_t\right) dx.
	\end{align*} 
Hence, using equation (\ref{NLS1}) and integrating by parts
	\begin{align*}
	\deri \left(\frac{1}{2}\intR \varphi|u(t)|^2 dx\right) 
	& = -\im \intR \varphi \overline{u} u_{xx} dx + \im \intR \varphi f(|u^2|)u \overline{u} dx \\
	& = \im \intR \varphi \overline{u}_x u_{x} dx+\im \intR \varphi_x \overline{u} u_{x} dx + \im \intR \varphi f(|u^2|) |u|^2 dx \\
	& = \im \intR \varphi |u_x|^2 dx + \im \intR \varphi_x \overline{u} u_{x} dx + \im \intR \varphi f(|u^2|) |u|^2 dx.
	\end{align*} 
Since the integrals on the first and third term are real, we get the following identity
	\begin{equation}\label{I1}
	\deri \left(\frac{1}{2}\intR \varphi|u(t)|^2\right)=\im \intR \varphi_x \overline{u} u_{x} dx.
	\end{equation}
Thus
	\begin{align*}
	\left|\deri \left(\frac{1}{2}\intR \varphi|u(t)|^2dx\right)\right|
	& \le \intR |\varphi_x| |\overline{u}(t)| |u_{x}(t)| dx\\
	&\lesssim \intR |\varphi_x| |\overline{u}(t)|^2 dx +\intR |\varphi_x| |u_{x}(t)|^2 dx.
	\end{align*} 
We take $\varphi(x)=\sech(x)$ and get
	\begin{align*}
	\left| \deri \|u(t)\|^2_{L^2_\alpha(\R)}\right|
	& = \left|\deri \left(\intR \sech(x)|u(t,x)|^2dx\right)\right| \\
	&\lesssim \intR \sech(x) |\overline{u}(t,x)|^2 dx +\intR \sech(x) |u_{x}(t,x)|^2 dx = \|u(t)\|^2_{H^1_\alpha(\R)}.
	\end{align*} 
\par From \eqref{key1}, there exists a sequence ${t_n} \in \R$, $t_n \to \infty$ such that $\|u(t_n)\|^2_{L^2_w(\R)}\to 0.$ Consider $t \in \R$, integrate over $[ t, t_n]$, and take $t_n \to \infty$. Then 
	\[\|u(t)\|^2_{L^2_\alpha(\R)}\lesssim \int_t^{\infty} \|u(s)\|^2_{H^1_\alpha(\R)} ds. \]
In consequence
	\begin{equation}\label{key2}
	\lim_{t\to \infty}\|u(t)\|_{L^2_w(\R)}=0.
	\end{equation}

\bigskip

\noindent \textit{Step 2: The $L^{\infty}$ norm tends to zero:} We state the following:
\bigskip 
\begin{claim}\label{claim1}
	For every interval $I$ there exists $\tilde{x}(t) \in I$ such that, as $t$ tends to infinity,
		\[|u(t, \tilde{x}(t))|^2\to 0.\]
\end{claim}
		
\begin{proof}
	Let $I\in \R$ be an interval. By contradiction, Suppose that there exists $\varepsilon_0>0$ such that $\forall n > 0$, $\exists t_n > n$ 
		\[|u(t_n, x)|^2> \varepsilon_0 \quad \forall x \in I .\]
	Integrating over $I$, we get 
		\[\int_I |u(t_n, x)|^2dx> |I| \varepsilon_0,\]
	wich contradicts ($\ref{key2}$).
\end{proof}	

Let $x \in I$. By Fundamental Theorem of calculus and H\"older's inequality 
	\begin{align*}
	|u(t,x)|^2 - |u(t, \tilde{x}(t))|^2 
	&= \int_{\tilde{x}(t)}^{x} \left(|u|^2\right)_x dx \le 2 \int_{\tilde{x}(t)}^{x} |u||u_x|dx\\
	& \le 2 \|u(t)\|_{L^2(I)}\|u_x(t)\|_{L^2(I)}.
	\end{align*}
Then we get
	\begin{equation}\label{f1}
	|u(t,x)|^2 \lesssim |u(t, \tilde{x}(t))|^2 + 2 \|u(t)\|_{L^2(I)}\|u_x(t)\|_{L^2(I)},\quad \forall x \in I.
	\end{equation}
Now, since \eqref{condition_s} holds for $\varepsilon>0$ as small as needed, 
	\[ \sup_{t \in \R} \|u(t)\|_{H^1(\R)}< \infty. 
	 \]
Also, this smallness condition is not needed if the nonlinearity is defocusing. Hence, taking $t \to \infty$ in \eqref{f1}, from Claim \ref{claim1} and \eqref{key2}, we get that 
	\[|u(t,x)|^2 \to 0, \quad \forall x \in I.\]
Which implies (\ref{result_s1}). The proof of Theorem \ref{T1} is complete.

\bigskip

\section{NLS with potential}\label{withpotential}

This section is devoted to the proof of Theorem \ref{T1a}. We consider now the NLS equation with a nontrivial potential $V$:
	\begin{equation}
	\label{NLS} iu_t+u_{xx}= \mu V(x)u+ f\left(|u|^2\right)u,  \quad (t, x) \in \mathbb{R}\times \mathbb{R}.
	\end{equation}
As done in the previous section, we introduce a virial identity that will be used to estimate the $H^1_\alpha$-norm of a solution of equation (\ref{NLS}). However, because of the potential term $V$, new estimates must be proved in order  to get Theorem \ref{T1a}.

\subsection{Virial Identity} 

Suppose again $\varphi \in C^{\infty}(\mathbb{R})$ bounded and recall from Subsection \ref{VIwithoutpotential} the definition 
	\begin{equation*}
	I(u(t))=\im\intR \varphi(x) u(t,x)\overline{u}_x(t,x)dx.
	\end{equation*}
Following the proof of Lemma \ref{VirialIdentity1}, we have now 

\begin{lemma} \label{VirialIdentity2}
	Let $u(t) \in H^1(\R)$ be a bounded in time solution of equation (\ref{NLS}). Then
		\begin{equation}\label{virial}
		\begin{aligned}
		-\deri I(t)= 
		&~{} 2 \intR \varphi_x |u_x|^2 dx - \frac{1}{2}\intR \varphi_{xxx} |u|^2 dx-\mu\intR \varphi V_x |u|^2dx \\
		& ~{} - \intR \varphi_x \left[F\left(|u|^2\right)-f(|u|^2)|u|^2\right] dx.
		\end{aligned}
		\end{equation}
	\end{lemma}

\begin{proof}[Sketch of proof]
	From the proof of Lemma \ref{VirialIdentity1} (equation \eqref{a1}) we know that	
		\begin{align*}
		\deri I(u(t)) = - 2\re \intR \varphi (i u_t) \overline{u}_x dx - \re \intR \varphi_x \left(i u_t\right)\overline{u} dx.
		\end{align*}
	We use (\ref{NLS}) to obtain
		\begin{align*}
		\deri I(u(t))=
		& ~{}2\re \intR \varphi u_{xx} \overline{u}_x dx + \re \intR \varphi_x  u_{xx} \overline{u} dx - 2\mu \re \intR \varphi Vu \overline{u}_x dx \\ 
		& ~{}-\mu \re \intR \varphi_x Vu \overline{u}dx - 2\re \intR \varphi f\left(|u|^2\right)u \overline{u}_x dx - \re \intR \varphi_x f\left(|u|^2\right)u\overline{u} dx.
		\end{align*}
	From the last equation, we are only interested in the terms involving the potential $V$, since the rest of them were analyzed in the proof of Lemma \ref{VirialIdentity1}. Then we compute
		\begin{align*}
		2 \re \intR \varphi Vu \overline{u}_x dx 
		+ \re \intR \varphi_x Vu \overline{u}dx
		&=\intR \varphi V\left(|u|^2 \right)_xdx+ \intR \varphi_x V|u|^2dx \\
		&=-\intR \varphi V_x|u|^2 dx.
		\end{align*}
Combining this with Lemma \ref{VirialIdentity1}, we conclude (\ref{virial}).
\end{proof}

\subsection{Analysis of a modified bilinear form}

Now we can see the differences between the cases with and without potential. In this occasion,  we define the bilinear form ($u=u_1+iu_2$, $u_i\in\R$)
	\[B(w)= 2 \intR \varphi_x {w}_x^2 dx - \frac{1}{2}\int _{\mathbb{R}} \varphi_{xxx} w^2 dx- \mu \intR \varphi V_xw^2 dx, \quad w=u_i,\quad  i=1,2.\]
		
Consider $\lambda \in (1, \infty)$, $\varphi(x)=\lambda \tanh \left(\frac{x}{\lambda}\right)$ and $\alpha(x)=\sqrt{\varphi_x(x)}$. Since $\alpha ^2 = \varphi_x$, we can write
	\begin{equation}\label{b3}
	\intR \varphi V_x w^2 dx = \intR V_x \frac{\varphi}{\varphi_x}\left(\alpha w\right)^2 dx.
	\end{equation}
Thus, from (\ref{b11}), (\ref{b3}) y (\ref{b12}), 
	\begin{align*}
	B(w)=2\intR \left(\alpha w\right)^2_x dx - \intR\left( \frac{\alpha_x^2}{\alpha^2}-\frac{\alpha_{xx}}{\alpha}\right)\left(\alpha w \right)^2 dx- \mu \intR V_x \frac{\varphi}{\varphi_x}\left(\alpha w\right)^2 dx.
	\end{align*}
Then, from computations of subsection \ref{bilinearanalisys1} we have that
	\begin{align*}
	B(w) = 2\intR \left(\alpha w\right)^2_x dx - \frac{1}{\lambda^2} \intR\sech^2\left(\frac{x}{\lambda}\right)\left(\alpha w \right)^2 dx- \mu \intR V_x \frac{\varphi}{\varphi_x}\left(\alpha w\right)^2 dx.
	\end{align*}
We set 
	\[\mathcal{B}(v) = 2\intR {v}^2_x dx - \frac{1}{\lambda^2} \intR\sech^2\left(\frac{x}{\lambda}\right)v^2 dx-\mu \intR V_x \frac{\varphi}{\varphi_x} v^2 dx,\]
where $v=\alpha w$. Then
	\[\mathcal{B}(v)=B(w).\]	
Now we prove a modified version of Proposition \ref{coercividad1}.		
		
\begin{proposition}\label{pcoercividad}
	Let $v \in H^1(\R)$ be odd. Then, for $\lambda > 0$ sufficiently small, 
		\begin{equation}\label{coercividad}
		\mathcal{B}(v)\ge \frac{1}{2} \intR v_x^2 dx.
		\end{equation}
\end{proposition}

\begin{proof}
	We introduce 
		\[\mathcal{L}(v) = \intR {v}^2_x dx - \frac{1}{\lambda^2} \intR\sech^2\left(\frac{x}{\lambda}\right)v^2 dx\]
	and 
		\[\mathcal{K}(v) = \intR {v}^2_x dx +\mu \intR V_0 v^2 dx, \]	
	where $V_0=- V_x \frac{\varphi}{\varphi_x}$.
	Then, 
		\[\mathcal{B}(v)=\mathcal{L}(v)+\mathcal{K}(v).\]
	Arguing as in the proof of Proposition \ref{coercividad1}, we write 
		\[\mathcal{L}(v) =\frac 12 \intR {v}^2_x dx + \frac 12 \left(\intR {v}^2_x dx- \frac{2}{\lambda^2} \intR\sech^2\left(\frac{x}{\lambda}\right)v^2 dx\right).\]
	Since $v$ is odd,
		\[\intR {v}^2_x dx- \frac{2}{\lambda^2} \intR\sech^2\left(\frac{x}{\lambda}\right)v^2 dx\ge 0,\]
	because the index $N$ of such an operator is the integer that satisfies $N<\frac 12 \sqrt{17}-\frac 12 <2 $.
	Hence, we get that
		\begin{equation}\label{coercividadL}
		\mathcal{L}(v)\ge \frac{1}{2} \intR v_x^2 dx.
		\end{equation} 
	Then, in order to get the (\ref{coercividad}) it will be sufficient to demonstrate that $\mathcal{K}(v)\ge 0$.
	
	\par To prove the positiveness of $\mathcal{K}$, we make use of the following result by Simon \cite[Theorem 2.5]{Simon} (see also \cite{Klaus} for improved results):

	\begin{lemma}\label{simon}
		Let $V_0$ be a non-identically zero potential that obeys \[\intR \left(1+x^2\right)|V_0(x)|dx < \infty.\] Then 
			\[-\frac{d^2}{dx^2}+\mu V_0 \]
		has a unique negative eigenvalue for all positive $\mu$ sufficiently small if and only if
			\begin{equation}\label{hipsimon}
			 \intR V_0(x)dx\le 0. 
			\end{equation}
			Moreover, since $V_0$ is even, such an eigenvalue is associated to an even eigenfunction.
	\end{lemma}	
	
	\begin{remark}
	We remark that in the case $\intR V_0>0$ there is no negative eigenvalue $-\frac{d^2}{dx^2}+\mu V_0$, $\mu>0$ sufficiently small.
	\end{remark}

	\par Notice that, from the definition of $\varphi$ and (\ref{hipotesisV}), we have 
		\begin{align*}\intR V_0 dx=-\intR V_x \frac{\varphi}{\varphi_x}dx=-\lambda\intR V_x \text{sinh}\left(\frac{x}{\lambda}\right)\text{cosh}\left(\frac{x}{\lambda}\right)dx
		\end{align*}
	We integrate by parts and get
		\begin{align*}
		\intR V_0 dx=\intR \cosh\left(\frac{2x}{\lambda}\right)V dx.
		\end{align*}
	Since $\lambda>1$, \eqref{hipotesisV} tells us that $V_0$ integrates in space.  
	Besides, since $V$ is a Schwartz function, 
		\[
		\intR \left(1+x^2\right)\left|V_x \frac{\varphi}{\varphi_x}\right|dx\le \intR \left(1+x^2\right)\left|V_x  \right| \cosh\left(\frac{2x}{\lambda}\right)dx< \infty.\]
	Then, Lemma \ref{simon} implies that there exists $\mu_0>0$ such that 
		\[-\frac{d^2}{dx^2}+\mu V_0 \]
	has a unique negative eigenvalue for all $\mu<\mu_0$ and $\lambda>1$. Since the corresponding eigenfunction is even, we have $\mathcal K(v)\geq 0$ for $v$ odd. 
\end{proof}

The conclusion that we obtain from Proposition \ref{coercividad} is that for $i=1, 2$, 
	\[B(u_i)\ge \frac12 \intR \left(\alpha u_i\right)_x^2dx.\]
This property of the bilineal form $B$ will allow us to get an estimation of the operator $\deri I(u(t))$ that will lead us to conclude the proof of Theorem \ref{T1a}.

\subsection{Estimates of the terms on (\ref{virial})}

\begin{lemma} \label{estimation3}
	Let $u$ be an odd solution of (\ref{NLS}). Then,
		\begin{equation}
		\|u\|^2_{H^1_\alpha(\R)} \le C\left(B(u_1)+B(u_2)\right).
		\end{equation}
	for some $C>0$.
\end{lemma}

\begin{proof} Direct from Lemma \ref{estimation2}.
\end{proof}

\begin{lemma} \label{estimation}
	There exists $\varepsilon>0$ such that for every odd solution $u$ of (\ref{NLS}) satisfying 
		\begin{equation}\label{hip}
		\|u(t)\|_{H^1(\mathbb{R})}\le \varepsilon \quad \forall t \in \mathbb{R},
		\end{equation}
	then
		\begin{equation}\label{cotaI}
		-\deri I (u(t))\ge C \|u(t)\|^2_{H^1_\alpha(\R)}.
		\end{equation}
	where $C>0$.
\end{lemma}
	
\begin{proof}
	The virial identity we have is
		\begin{equation*}
		\begin{aligned}
		-\deri I(t)= 
		&~{} 2 \intR \varphi_x |u_x|^2 dx - \frac{1}{2}\intR \varphi_{xxx} |u|^2 dx-\mu\intR \varphi V_x |u|^2dx \\
		& ~{} - \intR \varphi_x \left[F\left(|u|^2\right)-f(|u|^2)|u|^2\right] dx\\
		=&~{} B(u_1) + B(u_2) - \intR \varphi_x \left[F\left(|u|^2\right)-f(|u|^2)|u|^2\right] dx.
		\end{aligned}
		\end{equation*}
	As we already have an estimation for $B(u_1)+B(u_2)$ given by Lemma \ref{estimation3}, we need to check that the remaining terms can be controled. Replicating the proof of Lemma \ref{estimation1}, we get that 
		\begin{align*}
		 \intR \varphi_x \left[ F(|u|^2)-f(|u|^2)|u|^2 \right]dx 
		 & \lesssim \intR \sech^2\left(\frac{x}{\lambda}\right) |u|^{p+1}dx\\
		 & \lesssim \|u \|^{p-1}_{H^1(\R)} \intp \left|\left(\sech \left(\frac{x}{\lambda}\right)u_1\right)_x\right|^2 dx \\
		 & ~{} \quad + \|u \|^{p-1}_{H^1(\R)} \intp \left|\left(\sech \left(\frac{x}{\lambda}\right)u_2\right)_x\right|^2 dx.
		 \end{align*}	
	Thus, Proposition \ref{pcoercividad} implies that 
		\begin{align*}
		\intR \varphi_x \left[ F(|u|^2)-f(|u|^2)|u|^2 \right]dx 
		\lesssim \|u \|^{p-1}_{H^1(\R)} \big( B(u_1)+B(u_2) \big).
		\end{align*}	
	Now, since $\|u \|_{H^1(\R)}$ is small enough, we conclude. (In the defocusing case, this condition is not needed.)
\end{proof}	
	
We can modify the proof of Proposition \ref{key1}, using Lemma \ref{estimation} instead of Lemma \ref{estimation1} to obtain the following:

\begin{proposition}
	There exists a constant $C>0$ such that 
		\begin{equation}\label{key}
		\intp \|u(t)\|_{H^1_\alpha(\R)}^2 dt \le C\epsilon^2.
		\end{equation}
\end{proposition}

\subsection{Proof main result}

\bigskip 

\textit{Step 1: The $L^2$ norm tends to zero:}

\bigskip

Let $\varphi \in C^{\infty}(\R)$. 
Since $\im \intR \varphi V |u|^2 dx =0$, computing as in Subsection \ref{proof1}, we have 
	\begin{equation}\label{I22}
	\deri \left(\frac{1}{2}\intR \varphi|u(t)|^2\right)=\im \intR \varphi_x \overline{u} u_{x} dx
	\end{equation}
This identity implies that
	\begin{align*}
	\left|\deri \left(\frac{1}{2}\intR \varphi|u(t)|^2dx\right)\right|\lesssim \intR |\varphi_x| |\overline{u}(t)|^2 dx +\intR |\varphi_x| |u_{x}(t)|^2 dx.
	\end{align*} 
Taking $\varphi(x)=\sech(x)$ we obtain
	\begin{align*}
	\left| \deri \|u(t)\|^2_{L^2_\alpha(\R)}\right|
	& = \left|\deri \left(\intR \sech(x)|u(t)|^2\right)\right| \\
	&\lesssim \intR \sech(x) |\overline{u}(t,x)|^2 dx +\intR \sech(x) |u_{x}(t,x)|^2 dx = \|u(t)\|^2_{H^1_\alpha(\R)}.
	\end{align*} 
From (\ref{key}), there exists a sequence ${t_n} \in \R$, $t_n \to \infty$ such that $\|u(t_n)\|^2_{L^2_\alpha(\R)}\to 0.$
Consider $t \in \R$, integrate over $[t_n, t]$, and take $t_n \to \infty$. Then 
	\[\|u(t)\|^2_{L^2\alpha(\R)}\lesssim \int_t^{\infty} \|u(t)\|^2_{H^1_\alpha(\R)} dt. \]
Passing to the limit
	\begin{equation}\label{l2}
	\lim_{t\to \infty}\|u(t_n)\|_{L^2_\alpha(\R)}=0.
	\end{equation}
	
\bigskip

\noindent \textit{Step 2: The $L^{\infty}$ norm tends to zero:}

\bigskip 

One uses the same  arguments as in Subsection \ref{proof1}. We skip the proof. 
%
%

\bigskip

\section{The Hartree Equation. Proof of Theorem \ref{TH1}}\label{proof_Hartree}

Our goal in this section is to extend Theorem \ref{T1} to the Hartree equation,  
	\begin{equation}\label{Hartree}
	 iu_t+u_{xx}=\sigma \left( |x|^{-a} * |u|^2\right)u, \quad (t, x) \in \mathbb{R}\times \mathbb{R}
	 \end{equation}
where $\sigma=\pm 1$ and $0<a<1$. We start out with a virial identity.

\subsection{Virial Identity} 

As before (see \eqref{I}), let us consider $\varphi \in C^{\infty}(\mathbb{R})$ bounded and let 
	\begin{equation}\label{J}
	J(u(t)):=\im\intR \varphi(x) u(t,x)\overline{u}_x(t,x)dx, 
	\end{equation}
then we state the following result.

\begin{lemma}
	Let $u \in H^1(\R)$ be a solution of (\ref{Hartree}), then
		\begin{equation}\label{virialh}
		-	\deri J(u(t))= 2 \intR \varphi_x |u_x|^2 dx - \frac{1}{2}\intR \varphi_{xxx} |u|^2 dx  +\sigma a \intR \varphi \left(\frac{x}{|x|^{a+2}}*|u|^2\right) |u|^2 dx.
		\end{equation}
\end{lemma}

\begin{proof} 
	Recall \eqref{a1} from the proof of Lemma \ref{virial},
		\begin{align*}
		\deri J(u(t))= - 2\re \intR \varphi (i u_t) \overline{u}_x dx - \re \intR \varphi_x \overline{u} (i u_t) dx. 
		\end{align*}
	We use (\ref{Hartree}) to get
		\begin{align*}
		\deri J(u(t))
		&= 2\re \intR \varphi u_{xx} \overline{u}_x dx + \re \intR \varphi_x \overline{u} u_{xx} dx \\
		&~{}- \sigma 2\re \intR \varphi \left(|x|^{-a}*|u|^2\right)u \overline{u}_x dx - \sigma\re \intR \varphi_x \left(|x|^{-a}*|u|^2\right)u\overline{u} dx \\
	 	&=  \intR \varphi \left(|u_x|^2\right)_x dx + \re \intR \varphi_x \overline{u} u_{xx} dx\\
	 	&~{} -\sigma \intR \varphi \left(|x|^{-a}*|u|^2\right)\left(|u|^2\right)_x dx - \sigma \intR \varphi_x \left(|x|^{-a}*|u|^2\right)|u|^2 dx. 
	 	\end{align*}
	We integrate by parts once on the last term and twice on the second term to obtain
		\begin{align*}
		\deri J(u(t))
		&= -2 \intR \varphi_x |u_x|^2 dx - \re \intR \varphi_{xx} \overline{u} u_{x} dx  + \sigma \intR \varphi \left(|x|^{-a}*|u|^2\right)_x |u|^2 dx\\
		&= -2 \intR \varphi_x |u_x|^2 dx - \frac{1}{2}\intR \varphi_{xx} \left(|u|^2\right)_x dx + \sigma \intR \varphi \left(|x|^{-a}*|u|^2\right)_x |u|^2 dx\\
		&= -2 \intR \varphi_x |u_x|^2 dx + \frac{1}{2}\intR \varphi_{xxx} |u|^2 dx  + \sigma \intR \varphi \left(|x|^{-a}*|u|^2\right)_x |u|^2 dx.
		\end{align*}
	Computing the derivative on the last term, 
		\[	\deri J(u(t))= -2 \intR \varphi_x |u_x|^2 dx + \frac{1}{2}\intR \varphi_{xxx} |u|^2 dx  -\sigma a \intR \varphi \left(\frac{x}{|x|^{a+2}}*|u|^2\right)_x |u|^2 dx.\]
\end{proof}

Let us analyze the RHS of \eqref{virialh}. Notice that if $\varphi$ is a non-decreasing weight function, the integral on the last term in \eqref{virialh} is positive:
	 \begin{equation}\label{positivity}
	 \intR \intR \varphi(x) \frac{x-y}{|x-y|^{a+2}}|u(y)|^2|u(x)|^2 dy dx\ge 0.
	\end{equation} 
Indeed, we compute (all the computations below are justified by choosing suitably compactly supported functions, and taking the standard limit procedure)
	 \begin{align*}
	 \intR \varphi \left(|x|^{-a}*|u|^2\right)_x |u|^2 dx=&-a\intR \varphi \left(\frac{x}{|x|^{a+2}}*|u|^2\right)|u|^2 dx \\
	 & =-a \intR \intR \varphi(x) \frac{x-y}{|x-y|^{a+2}}|u(y)|^2|u(x)|^2 dy dx .
	 \end{align*} 
We have that 
	 \begin{align*}
	 \intR \intR \varphi(x) \frac{x-y}{|x-y|^{a+2}}|u(y)|^2|u(x)|^2 dy dx 
	& =  \intR \intR\left( \varphi(x)-\varphi(y)\right) \frac{x-y}{|x-y|^{a+2}}|u(y)|^2|u(x)|^2 dy dx\\
 	& ~{} + \intR \intR \varphi(y) \frac{x-y}{|x-y|^{a+2}}|u(y)|^2|u(x)|^2 dy dx.
	\end{align*} 
After a change of variables on the second integral, we get
	 \begin{align*}
	 \intR \intR \varphi(x) \frac{x-y}{|x-y|^{a+2}}|u(y)|^2|u(x)|^2 dy dx 
	  =&  \intR \intR\left( \varphi(x)-\varphi(y)\right) \frac{x-y}{|x-y|^{a+2}}|u(y)|^2|u(x)|^2 dy dx\\
	 & ~{} - \intR \intR \varphi(x) \frac{x-y}{|x-y|^{a+2}}|u(y)|^2|u(x)|^2 dy dx.
	 \end{align*} 
 Then, we obtain that
 	\begin{align*}
	\intR \intR \varphi(x) \frac{x-y}{|x-y|^{a+2}}|u(y)|^2|u(x)|^2 dy dx = \frac{1}{2} \intR \intR\left( \varphi(x)-\varphi(y)\right) \frac{x-y}{|x-y|^{a+2}}|u(y)|^2|u(x)|^2 dy dx.
	\end{align*} 
If $\varphi$ is non-decreasing, then $\left(\varphi(x)-\varphi(y)\right) (x-y)\ge 0$. Moreover, 
	 \[\intR \intR\left( \varphi(x)-\varphi(y)\right) \frac{x-y}{|x-y|^{a+2}}|u(y)|^2|u(x)|^2 dy dx\ge 0.
	 \]
 This implies that 
 	\[ \intR \intR \varphi(x) \frac{x-y}{|x-y|^{a+2}}|u(y)|^2|u(x)|^2 dy dx\ge 0,\]
	as claimed.

\subsection{Proof of Theorem \ref{TH1}}

Assume $\sigma=1$ in \eqref{Hartree} and let $u=u_1+iu_2\in H^1(\R)$ be an odd solution of this equation. As done in Section \ref{withoutpotential}, we define the bilinear form
	\[B(u_i)=2\intR \varphi_x {u_i}_x^2 dx - \frac{1}{2} \intR\varphi_{xxx} u_i^2 dx, \quad i=1,2.\]
This means that we can re-write the virial identity \eqref{J} as 	
	\begin{align}\label{J2}
	-\deri J(u(t)) = B(u_1)+B(u_2) - \sigma \intR \varphi \left(|x|^{-a}*|u|^2\right)_x |u|^2 dx.	
	\end{align}
Now, as usual, take $\lambda>1$, $\varphi=\lambda \tanh\left(\frac{x}{\lambda}\right)$ and $\alpha=\sqrt{\varphi_x}$. From (\ref{b11}) and (\ref{b12}) and reasoning as before, we have that 
	\[B(u_i)=2\intR \left(\alpha u_i\right)^2_x dx - \frac{1}{\lambda^2} \intR\sech^2\left(\frac{x}{\lambda}\right)\left(\alpha u_i \right)^2 dx, \quad i=1,2.\]	
Thus, Proposition \ref{coercividad1} implies that
	\begin{align}\label{g1}
	B(u_i) \ge \frac{3}{2} \intR \left(\alpha u_i\right)_x^2 dx, \quad \text{for } i=1,2.
	\end{align}
Moreover, if we consider 
	\[\|u(t)\|_{H^1_\alpha(\R)}=\intR \sech(x)u^2(t,x)dx+ \intR \sech(x)u_x^2(t,x) dx,
	\]
then, from Proposition \ref{estimation2} we obtain
	\begin{equation}\label{g2}
	\|u\|_{H^1_\alpha(\R)}^2\lesssim B(u_1) + B(u_2).
	\end{equation}
Since 
 	\[ \intR \intR \varphi(x) \frac{x-y}{|x-y|^{a+2}}|u(y)|^2|u(x)|^2 dy dx\ge 0,
	\]
it follows that
	\[-\deri J(u(t)) \ge \|u\|_{H_\alpha^1(\R)}^2.\]
Replicating the proof of Proposition \ref{propositionkey}, we use the last inequality to obtain
	\begin{equation}\label{key4}
	\intp \|u(t)\|_{H^1_\alpha(\R)}^2 dt \le C \varepsilon^2.
	\end{equation}	
	
	\noindent
\textit{Step 1: The $L^2$ norm tends to zero:} Let $\phi \in C^{\infty}(\R)$ bounded. Then we compute
	\begin{align*}
	\deri \left(\frac{1}{2}\intR \phi|u(t)|^2 dx\right) 
	&=\re \intR \phi\overline{u} u_t dx \\
	&=-\re \intR i \phi\overline{u}  \left(iu_t\right)  dx\\
	&= \im \intR \phi \overline{u} \left(iu_t\right) dx.
	\end{align*} 
Hence, using equation (\ref{Hartree}) with $\sigma=1$ and integrating by parts
	\begin{align*}
	\deri \left(\frac{1}{2}\intR \phi|u(t)|^2 dx\right) 
	& = -\im \intR \phi \overline{u} u_{xx} dx + \im \intR \phi \left( |x|^{-a}* |u|^2 \right)u \overline{u} dx \\
	& = \im \intR \phi \overline{u}_x u_{x} dx+\im \intR \phi \left( |x|^{-a}* |u|^2 \right)|u|^2 dx\\
	& = \im \intR \phi |u_x|^2 dx + \im \intR \phi_x \overline{u} u_{x} dx + \im \intR \phi \left( |x|^{-a}* |u|^2 \right)|u|^2 dx.
	\end{align*} 
Since the only integral that can have an imaginary part is the second one, we have that
	\begin{equation}\label{h1}
	\deri \left(\frac{1}{2}\intR \phi|u(t)|^2\right)=\im \intR \phi_x \overline{u} u_{x} dx
	\end{equation}
Thus
	\begin{align*}
	\left|\deri \left(\frac{1}{2}\intR \phi|u(t)|^2dx\right)\right|
	& \le \intR |\phi_x| |\overline{u}(t)| |u_{x}(t)| dx\\
	&\lesssim \intR |\phi_x| |\overline{u}(t)|^2 dx +\intR |\phi_x| |u_{x}(t)|^2 dx.
	\end{align*} 
We take $\phi(x)=\sech(x)$ and get
	\begin{align*}
	\left| \deri \|u(t)\|^2_{L^2_\alpha(\R)}\right|
	& = \left|\deri \left(\intR \sech(x)|u(t)|^2\right)\right| \\
	&\lesssim \intR \sech(x) |\overline{u}(t,x)|^2 dx +\intR \sech(x) |u_{x}(t,x)|^2 dx = \|u(t)\|^2_{H^1_\alpha(\R)}.
	\end{align*} 
\par From \eqref{key4}, there exists a sequence ${t_n} \in \R$, $t_n \to \infty$ such that $\|u(t_n)\|^2_{L^2_\alpha(\R)}\to 0.$ Consider $t \in \R$, integrate over $[ t, t_n]$, and take $t_n \to \infty$. Then 
	\[\|u(t)\|^2_{L^2_\alpha(\R)}\lesssim \int_t^{\infty} \|u(s)\|^2_{H^1_\alpha(\R)} ds. \]
In consequence
	\begin{equation}\label{key5}
	\lim_{t\to \infty}\|u(t)\|_{L^2_\alpha(\R)}=0.
	\end{equation}
The rest of the proof is exactly the same as in the proofs of Theorems \ref{T1} and \ref{T1a}.

\bigskip

%
%
%


%


\begin{thebibliography}{99} 

	\bibitem{AFM} M.A. Alejo, L. Fanelli, and C. Mu\~noz, \emph{Stability and instability of breathers in the $U(1)$ Sasa-Satsuma and Nonlinear Schr\"odinger models}, preprint arXiv: 1901.10381 (2019). 
	
	\bibitem{AR} A. K. Arora, and S. Roudenko, \emph{Global behavior of solutions to the focusing generalized Hartree equation}, preprint arXiv:1904.05339v1 (2019).
	
	\bibitem{Barab}J. E.  Barab \emph{Nonexistence of asymptotically free solutions for a nonlinear Schro\"odinger equation}, Journal of Mathematical Physics 25, 3270 (1984); doi: 10.1063/1.526074.
		
	\bibitem{BuslaevPerelman} V. S. Buslaev, and G. Perelman, \emph{Nonlinear scattering: the states that are closed to a soliton}, Journal of Mathematical Sciences, Vol. 77, no. 3 (1995), pp. 3161--3169.
	
	\bibitem{BuslaevSulem} V. S. Buslaev, and C. Sulem, \emph{On asymptotic stability of solitary waves for the nonlinear Schr\"odinger equations}, Annales de l'Institut Henri Poincar\'e (C), Vol. 20, Issue 3, pp. 419--475.
	
	\bibitem{Cazenave} T. Cazenave. \emph{Semilinear Schr\"odinger equations}, Courant Lecture Notes in Mathematics, New York University, New York, 2003.
	
	\bibitem{CL} T. Cazenave and P.-L. Lions, \emph{Orbital stability of standing waves for some nonlinear Schr\"odinger equations},  Comm. Math. Phys.  85 (1982), no. 4, 549--561.
	
	\bibitem{CW} T. Cazenave and F. Weissler, \emph{The Cauchy problem for the critical nonlinear Schr\"odinger equation in} $H^s$, Nonlinear Anal. 14 (1990), no. 10, 807--836.

	
	\bibitem{CSS} S. Cingolani, S. Secchi and M. Squassina, \emph{Semi-classical limit for Schr\"odinger equations with magnetic field and Hartree-type nonlinearities}., Proceedings of the Royal Society of Edinburgh: Section A Mathematics, 140(5) (2017), pp. 973--1009. 
	
	\bibitem{Cucc1} S. Cuccagna (2001), \emph{Stabilization of solutions to nonlinear Schr\"odinger equations}, Comm. Pure Appl. Math., 54 (2001), pp. 1110--1145. 

	\bibitem{Cucc2} S. Cuccagna, \emph{The Hamiltonian Structure of the Nonlinear Schr\"odinger Equation and the Asymptotic Stability of its Ground States}, Commun. Math. Phys. 305 (2011), pp. 279--331. 
	
	\bibitem{CuccagnaMaeda} S. Cuccagna, and M. Maeda \emph{On stability of small solitons of the 1--D NLS with a trapping potential}, arXiv:1904.11869.
	
	\bibitem{CM} S. Cuccagna and T. Mizumachi, \emph{On Asymptotic Stability in Energy Space of Ground States for Nonlinear Schro\"odinger Equations}, Commun. Math. Phys. 284 (2008), pp. 51--77. 
	
	\bibitem{CVG} S. Cuccagna, N. Visciglia and V. Georgiev, \emph{Decay and scattering of small solutions of pure power NLS in $\R$ with $p > 3$ and with a potential}, Comm. Pure Appl. Math, 67 (2014), pp. 957--981. 
	
	\bibitem{DZ} P. Deift, and X. Zhou, \emph{Perturbation theory for the infinite-dimensional integrable system on the line. A case study}, Acta Math, Vol. 188, Number 2 (2002), pp. 163--262.
	
	\bibitem{HN} N. Hayashi and  P. I. Naumkin, \emph{Asymptotics for large time of solutions to the nonlinear Schr\"odinger and Hartree equations}, Amer. J. Math. 120 (1998), pp. 369--389.
	
	\bibitem{Delort} J.-M. Delort, \emph{Modified scattering for odd solutions of cubic nonlinear Schr\"odinger equations with potential in dimension one}, preprint 2016 $<\text{hal}-01396705>$.
	
	\bibitem{GPR0} P. Germain, F. Pusateri and F. Rousset, \emph{Asymptotic stability of solitons for mKdV}, Adv. Math. 299 (2016), pp. 272--330.
	
	\bibitem{GPR} P. Germain, F. Pusateri and F. Rousset, \emph{The nonlinear Schr\"odinger equation with a potential}, Annales de l'{}Institut Henri Poincar\'e (C), 35 (2018), pp. 1477--1530. 
	
	\bibitem{GO} J. Ginibre and T. Ozawa, \emph{Long Range Scattering
	for Non-Linear Schrodinger and Hartree Equations in Space Dimension $n>2$}, Commun. Math. Phys. 151 (1993), pp. 619--645.
	
	\bibitem{GinibreVelo} J. Ginibre, and G. Velo, \emph{On a class of nonlinear Schr\"odinger equations. I: The Cauchy problem}, J. Funct. Anal. 32 (1979), pp. 1--32. 
	
	\bibitem{Glassey} R. Glassey, \emph{On the blowing up of solutions to the Cauchy problem for nonlinear Schrödinger equations}, J. Math. Phys. 18 (1977), no. 9, pp. 1794--1797.

	\bibitem{Goldman} I. I. Gol'dman, V. D. Krivchenkov, B. T. Geilikman, E. Marquit, E. Lepa, \emph{Problems in quantum mechanics}, Authorised revised ed. Edited by B. T. Geilikman; translated from the Russian by E. Marquit and E. Lepa, Pergamon Press, 1961.
	
	\bibitem{GSS} M. Grillakis, J. Shatah, and W. Strauss, \emph{Stability theory of solitary waves in the presence of symmetry. I}.  J. Funct. Anal.  74  (1987),  no. 1, 160--197.
	
	\bibitem{Kap} T. Kapitula, \emph{On the stability of $N$-solitons in integrable systems}, Nonlinearity, 20 (2007) pp. 879--907.

	 \bibitem{KP} J. Kato and F. Pusateri, \emph{A new proof of long range scattering for critical nonlinear Schr\"odinger equations}, Differential Integral Equations 24 (2011), pp. 923--940.
	
	\bibitem{Klaus} M. Klaus, \emph{On the Bound State of Schr\"odinger Operators in One Dimension}, Annals of Physics 108 (1977), pp. 288--300.
	
	\bibitem{KMM1} M. Kowalczyk, Y. Martel and C. Mu\~noz, \emph{Kink dynamics in the $\phi^4$ model: asymptotic stability for odd perturbations in the energy space}, J. Amer. Math. Soc. 30 (2017), 769--798. 
	
	\bibitem{KMM2} M. Kowalczyk, Y. Martel and C. Mu\~noz, \emph{Nonexistence of small, odd breathers for a class of nonlinear wave equations}, Lett. Math. Phys. 107 (2017), pp. 921--931. 
	
	\bibitem{KMM3} M. Kowalczyk, Y. Martel and C. Mu\~noz, \emph{Soliton dynamic for the 1D NLKG equation with symmetry and in the absence of internal modes}, preprint: arXiv: 1903.12460.
	
	\bibitem{LiebLoss} E. Lieb and M. Loss, \emph{Analysis}, 2nd. Ed., Graduate Studies in Mathematics, Vol. 14, Amer. Math. Soc., 2001. 
	
	
		
	\bibitem{MManihp} Y. Martel, and F. Merle, \emph{Multi-solitary waves for nonlinear Schr\"odinger equations}, Annales de l'Institut Henri Poincar\'e (C) 23 (2006), 849--864.

	\bibitem{MMT2} Y. Martel, F. Merle and T.-P. Tsai, \emph{Stability in $H^1$ of the sum of $K$ solitary waves for some nonlinear Schr\"odinger equations}, Duke Math. J. 133 (2006), pp. 405--466.
	
	\bibitem{MR} F. Merle and P. Rapha\"el, \emph{The blow-up dynamic and upper bound on the blow-up rate for critical nonlinear {S}chr\"odinger equation}, Ann. of Math. (2) \textbf{161} (2005), no.~1, 157--222.
	
	\bibitem{VanSchaftingen} V. Moroz and J. Van Schaftingen, \emph{A guide to the Choquard equation}, J. Fixed Point Theory Appl. 19 (2017), pp. 773--813.
	
	\bibitem{MP} C. Mu\~noz and G. Ponce, \emph{Breathers and the dynamics of solutions to the KdV type equations}, Comm. Math. Phys. April 2019, Volume 367, Issue 2, pp 581--598.
	
	\bibitem{MurphyNakanishi} J. Murphy, and K. Nakanishi, \emph{Failure of scattering to solitary waves for the long-range nonlinear Schr\"odinger equations}, preprint arXiv:  1906.01802.
	
	\bibitem{NO} K. Nakanishi, and T. Ozawa, \emph{Remarks on scattering for nonlinear Schr\"odinger equations}, NoDEA 2002 Vol. 9, Issue 1, pp. 45--68.
	
	\bibitem{Naumkin} I. P. Naumkin, \emph{Sharp asymptotic behavior of solutions for cubic nonlinear Schr\"odinger equations with a potential}, J. Math. Phys. 57 (2016), pp. 051501.
	
	\bibitem{Vinh} T. V. Nguyen, \emph{Existence of multi-solitary waves with logarithmic relative distances for the NLS equation},
	Comptes Rendus Mathematique,
	357 (2019), pp. 13--58.
	
	\bibitem{Ozawa} T. Ozawa, \emph{Long range scattering for nonlinear Schr\"odinger
		equations in one space dimension}, Comm. Math. Phys. 139 (1991), no. 3, 479--493. 
	

	\bibitem{Ruiz} D. Ruiz, \emph{On the Schr\"odinger-Poisson-Slater System: Behavior of Minimizers, Radial and Nonradial Cases}, Arch. Rational. Mech. Anal. 198 (2010), pp. 349--368. 
	
	\bibitem{SY} J. Satsuma, and N. Yajima, \emph{Initial Value Problems of One-Dimensional Self-Modulation of Nonlinear Waves in Dispersive Media}, Supplement of the Progress of Theoretical Physics, No. 55, 1974, 284--306.

	
	\bibitem{Simon} B. Simon. \emph{The Bound State of Weakly Coupled Schr\"odinger Operators in One and Two Dimensions},  Annals of Physics 97 (1976), pp. 279--288.
	
	\bibitem{SW} A. Soffer, and M. I. Weinstein, \emph{Resonances, radiation damping and instability in Hamiltonian nonlinear wave equations}, Invent. Math., March 1999, Volume 136, Issue 1, pp. 9--74.
	
	\bibitem{SW1} A. Soffer, and M. I. Weinstein, \emph{Multichannel Nonlinear Scattering for Nonintegrable Equations}  Commun. Math. Phys. 133, (1990), pp, 119--146.
		
	\bibitem{SW2} A. Soffer, and M. I. Weinstein, \emph{ Multichannel Nonlinear Scattering for Nonintegrable Equations II. The Case of Anisotropic Potentials and Data}, Journal of Differential Equations, Vol. 98, Issue 2 (1992), pp. 376--390.
	
	\bibitem{Strauss} W. Strauss, \emph{Nonlinear Scattering Theory}, Scattering Theory in Mathematical Physics, edited by J. A. Lavita and J-P. Marchand (Reidel, Dordrecht, Holland, 1974), pp. 53--78.  
	
	
	\bibitem{Tsutsumi0} Y. Tsutsumi, \emph{Scattering problem for nonlinear Schr\"odinger equations}, Annales de l'Institut Henri Poincar\'e, section A, tome 43, no 3 (1985), pp. 321--347.
	
	\bibitem{Tsutsumi} Y. Tsutsumi, \emph{$L^2$-solutions for nonlinear Schr\"odinger equations and nonlinear groups}, Funkcial. Ekvac. 30 (1987), no. 1, pp. 115--125.
	
	\bibitem{Weinstein} M. I. Weinstein, \emph{Lyapunov stability of ground states of nonlinear dispersive evolution equations}, Comm. Pure Appl. Math. \textbf{39}, (1986) pp. 51--68.

	\bibitem{ZS} V. E. Zakharov, and A. B. Shabat,  \emph{Exact theory of two-dimensional self-focusing and one-dimensional self-modulation of waves in nonlinear media}, JETP, 34 (1): pp. 62--69.
		
\end{thebibliography}
\end{document}